\documentclass[a4paper,11pt]{amsart}

\usepackage[utf8]{inputenc}
\usepackage[english]{babel}
\usepackage[bitstream-charter]{mathdesign}
\usepackage[T1]{fontenc}
\usepackage{amsfonts}
\usepackage{geometry}
\usepackage{amsmath, amsthm}
\usepackage[colorlinks = true]{hyperref}
\usepackage{color}
\usepackage{pgf,tikz}
\usetikzlibrary{arrows}
\usepackage{cancel}
\usepackage{color}
\usepackage{subfigure}
\usepackage{float}
\usepackage{verbatim}
\usepackage{wrapfig}
\usepackage{mathtools}

\usepackage[normalem]{ulem}

\newcommand\rmi{\hbox{\rm (i)}}
\newcommand\rmii{\hbox{\rm (ii)}}

\usepackage{pdfsync}
\usepackage{color}
\usepackage{mathrsfs}

\usepackage[T1]{fontenc}
\usepackage{mathrsfs}
\usepackage{epsfig}
\usepackage{xypic}
\usepackage{verbatim}
\usepackage{multicol}
\usepackage{wrapfig}
\usepackage{pgf,tikz}
\usepackage{mathrsfs}
\usetikzlibrary{arrows}

\numberwithin{equation}{section}

	\usepackage{amsrefs}

\renewcommand{\Re}{\mathrm{Re}}

\newcommand{\Ran}{\mathrm{Ran}}
\newcommand{\Dom}{\mathrm{Dom}}

\newcommand{\N}{\mathbb{N}}

\newcommand{\R}{\mathbb{R}}
\newcommand{\C}{\mathbb{C}}

\usepackage{mathrsfs}
\usepackage{mathtools}
\usepackage{comment}
\usepackage{indentfirst}
\usepackage{braket}
\usepackage{scalerel}

\DeclareMathOperator{\supp}{supp}

\DeclarePairedDelimiter{\abs}{\lvert}{\rvert}

{\left\lbrace\begin{array}{@{}l@{}}}%
{\end{array}\right.}

\makeatletter
\newcommand*{\mint}[1]{%
	\mint@l{#1}{}%
}
\newcommand*{\mint@l}[2]{%
	\@ifnextchar\limits{%
		\mint@l{#1}%
	}{%
	\@ifnextchar\nolimits{%
		\mint@l{#1}%
	}{%
	\@ifnextchar\displaylimits{%
		\mint@l{#1}%
	}{%
	\mint@s{#2}{#1}%
}%
}%
}%
}
\newcommand*{\mint@s}[2]{%
	\@ifnextchar_{%
		\mint@sub{#1}{#2}%
	}{%
	\@ifnextchar^{%
		\mint@sup{#1}{#2}%
	}{%
	\mint@{#1}{#2}{}{}%
}%
}%
}
\def\mint@sub#1#2_#3{%
	\@ifnextchar^{%
		\mint@sub@sup{#1}{#2}{#3}%
	}{%
	\mint@{#1}{#2}{#3}{}%
}%
}
\def\mint@sup#1#2^#3{%
	\@ifnextchar_{%
		\mint@sub@sup{#1}{#2}{#3}%
	}{%
	\mint@{#1}{#2}{}{#3}%
}%
}
\def\mint@sub@sup#1#2#3^#4{%
	\mint@{#1}{#2}{#3}{#4}%
}
\def\mint@sup@sub#1#2#3_#4{%
	\mint@{#1}{#2}{#4}{#3}%
}
\newcommand*{\mint@}[4]{%
	\mathop{}%
	\mkern-\thinmuskip
	\mathchoice{%
		\mint@@{#1}{#2}{#3}{#4}%
		\displaystyle\textstyle\scriptstyle
	}{%
	\mint@@{#1}{#2}{#3}{#4}%
	\textstyle\scriptstyle\scriptstyle
}{%
\mint@@{#1}{#2}{#3}{#4}%
\scriptstyle\scriptscriptstyle\scriptscriptstyle
}{%
\mint@@{#1}{#2}{#3}{#4}%
\scriptscriptstyle\scriptscriptstyle\scriptscriptstyle
}%
\mkern-\thinmuskip
\int#1%
\ifx\\#3\\\else_{#3}\fi
\ifx\\#4\\\else^{#4}\fi  
}
\newcommand*{\mint@@}[7]{%
	\begingroup
	\sbox0{$#5\int\m@th$}%
	\sbox2{$#5\int_{}\m@th$}%
	\dimen2=\wd0 %
	\let\mint@limits=#1\relax
	\ifx\mint@limits\relax
	\sbox4{$#5\int_{\kern1sp}^{\kern1sp}\m@th$}%
	\ifdim\wd4>\wd2 %
	\let\mint@limits=\nolimits
	\else
	\let\mint@limits=\limits
	\fi
	\fi
	\ifx\mint@limits\displaylimits
	\ifx#5\displaystyle
	\let\mint@limits=\limits
	\fi
	\fi
	\ifx\mint@limits\limits
	\sbox0{$#7#3\m@th$}%
	\sbox2{$#7#4\m@th$}%
	\ifdim\wd0>\dimen2 %
	\dimen2=\wd0 %
	\fi
	\ifdim\wd2>\dimen2 %
	\dimen2=\wd2 %
	\fi
	\fi
	\rlap{%
		$#5%
		\vcenter{%
			\hbox to\dimen2{%
				\hss
				$#6{#2}\m@th$%
				\hss
			}%
		}%
		$%
	}%
	\endgroup
}

\makeatletter
\def\overbracket#1{\mathop{\vbox{\ialign{##\crcr\noalign{\kern3\p@}
				\downbracketfill\crcr\noalign{\kern3\p@\nointerlineskip}
				$\hfil\displaystyle{#1}\hfil$\crcr}}}\limits}
\def\underbracket#1{\mathop{\vtop{\ialign{##\crcr
				$\hfil\displaystyle{#1}\hfil$\crcr\noalign{\kern3\p@\nointerlineskip}
				\upbracketfill\crcr\noalign{\kern3\p@}}}}\limits}
\def\overparenthesis#1{\mathop{\vbox{\ialign{##\crcr\noalign{\kern3\p@}
				\downparenthfill\crcr\noalign{\kern3\p@\nointerlineskip}
				$\hfil\displaystyle{#1}\hfil$\crcr}}}\limits}
\def\underparenthesis#1{\mathop{\vtop{\ialign{##\crcr
				$\hfil\displaystyle{#1}\hfil$\crcr\noalign{\kern3\p@\nointerlineskip}
				\upparenthfill\crcr\noalign{\kern3\p@}}}}\limits}
\def\downparenthfill{$\m@th\braceld\leaders\vrule\hfill\bracerd$}
\def\upparenthfill{$\m@th\bracelu\leaders\vrule\hfill\braceru$}
\def\upbracketfill{$\m@th\makesm@sh{\llap{\vrule\@height3\p@\@width.7\p@}}%
	\leaders\vrule\@height.7\p@\hfill
	\makesm@sh{\rlap{\vrule\@height3\p@\@width.7\p@}}$}
\def\downbracketfill{$\m@th
	\makesm@sh{\llap{\vrule\@height.7\p@\@depth2.3\p@\@width.7\p@}}%
	\leaders\vrule\@height.7\p@\hfill
	\makesm@sh{\rlap{\vrule\@height.7\p@\@depth2.3\p@\@width.7\p@}}$}
\makeatother

\theoremstyle{theorem}
\newtheorem{theorem}{\sc Theorem}[section]  
\newtheorem{proposition}[theorem]{\sc Proposition}   
        
\newtheorem{lemma}[theorem]{\sc Lemma}

\theoremstyle{remark}
\newtheorem{definition}[theorem]{\sc Definition}

\newtheorem{remark}[theorem]{Remark}

\usepackage{amsfonts}

\newcommand{\loc}{\mathrm{loc}}
\newcommand{\glob}{\mathrm{glob}}

\newcommand{\dd}{\mathrm{d}}

\newcommand{\Dir}{\mathrm{Dir}}

\newcommand{\Ls}{\mathcal{L}}

\newcommand{\Pc}{\mathcal{P}}

\newcommand{\Addresses}{Dipartimento di Scienze Matematiche ``Giuseppe Luigi Lagrange'', Politecnico di Torino, Corso Duca degli Abruzzi 24, 10129 Torino, Italy\\
\textit{E-mail address}: \texttt{tommaso.bruno@polito.it} 
}

\date{}

\begin{document}
\title[The Riesz transform of the Ornstein-Uhlenbeck~operator]{Endpoint results for the Riesz transform of the Ornstein-Uhlenbeck~operator}

\author{Tommaso Bruno}

	\maketitle
	\begin{abstract}\noindent
In this paper we introduce a new atomic Hardy space $X^1(\gamma)$ adapted to the Gauss measure $\gamma$, and prove the boundedness of the first order Riesz transform associated with the Ornstein-Uhlenbeck operator from $X^1(\gamma)$ to $L^1(\gamma)$. We also provide a new, short and almost self-contained proof of its weak-type $(1,1)$.

\bigskip


\smallskip


	\end{abstract}
	
\section{Introduction}\label{notation}
For $x\in \R^n$, let $\dd \gamma(x)=\pi^{-n/2}e^{-|x|^2}\, \dd x$ be the Gauss measure and denote with $\mathcal{L}$ the Ornstein-Uhlenbeck operator, i.e.\ the closure  on $L^2(\gamma)$ of the operator given by 
\[ -\tfrac{1}{2}\Delta  + x\cdot \nabla \]
on the space $C_c^\infty$ of smooth and compactly supported functions. It is well known that $\mathcal{L}$ is self-adjoint. We denote by $\nabla \Ls^{-1/2}$ its first order Riesz transform, which can be defined on $L^2(\gamma)$ via the spectral theorem (see Section~\ref{intKernels} below).

For every $p\in (1,\infty)$, the operator $\nabla \Ls^{-1/2}$ extends to a bounded operator on $L^p(\gamma)$, but this fails when $p=1$ (see e.g.~\cite{MMS0} or~\cite{Sjogren}). This motivates the interest in boundedness results involving $L^1(\gamma)$, which we call endpoint results, for this operator. Concerning boundedness properties \emph{from} $L^1(\gamma)$, the following result is well known:
\begin{theorem}\label{11OU}
$\nabla \Ls^{-1/2}$ is of weak type $(1,1)$, i.e.\ bounded from $L^1(\gamma)$ to $L^{1,\infty}(\gamma)$.
\end{theorem}
The proof of this result when $n=1$ is due to Muckenhoupt~\cite{Mucken}; in arbitrary dimension to Fabes, Gutiérrez and Scotto~\cite{FGS}. A new proof of this fact, shorter but still rather involved, was given by Pérez and Soria~\cite{PerSor} who used related results of Pérez~\cite{Per} and Men\'arguez, Pérez and Soria~\cite{MenPerSor}. 

The question of finding a subspace of $L^1(\gamma)$ mapped by $\nabla\Ls^{-1/2}$  \emph{into} $L^1(\gamma)$ has been considered more recently. In the pioneering paper~\cite{MM}, Mauceri and Meda introduced an atomic Hardy space $H^1(\gamma)$ adapted to the Gauss measure and studied boundedness properties of certain singular integral operators associated with $\Ls$ from this space to $L^1(\gamma)$. Among other results, they proved that the imaginary powers $\Ls^{iu}$ and the adjoint Riesz transform $\Ls^{-1/2}\nabla^*$ are bounded from $H^1(\gamma)$ to $L^1(\gamma)$. A few years later, however, the same authors and Sj\"ogren~\cite{MMS} proved that, though the Riesz transform $\nabla \Ls^{-1/2}$ is bounded from $L^\infty$ to the dual of $H^1(\gamma)$ in any dimension, it is bounded from $H^1(\gamma)$ to $L^1(\gamma)$ if and only if $n=1$. The problem of finding an appropriate subspace of $L^1(\gamma)$ mapped boundedly to $L^1(\gamma)$ by $\nabla\Ls^{-1/2}$ was addressed by Portal~\cite{Portal} who introduced a new Gaussian Hardy space $h^1(\gamma)$ and proved that $\nabla \Ls^{-1/2}$ is bounded from $h^1(\gamma)$ to $L^1(\gamma)$. The Hardy space $h^1(\gamma)$ is defined equivalently either by conical square functions or by a maximal function.

Portal's proof hinges on a theory of tent spaces for the Gauss measure developed by the same author and Maas and Van Nerven~\cite{MaasVNPortal2}. Though the tent spaces introduced in~\cite{MaasVNPortal2} admit an atomic decomposition and Portal's space is defined as a retract of a tent space via a Calder\'on reproducing formula, an explicit atomic characterization of $h^1(\gamma)$ is not provided in~\cite{Portal}.  This is our main motivation to explore a different approach to the problem, which we present in the first part of this paper. Indeed, atomic decompositions are a useful tool to prove boundedness of linear operators: in many circumstances (see e.g.~\cite{MM0, MSV}), it is enough to check that an operator maps atoms boundedly in some target space $Y$ to extend it to a bounded operator from the whole atomic space to $Y$. Inspired by the work of Mauceri, Meda and Vallarino~\cite{MMVAtomic} for the Riesz transforms on certain noncompact manifolds of infinite volume, we introduce a new \emph{atomic} Gaussian Hardy space $X^1(\gamma)$, strictly contained in the space $H^1(\gamma)$ of Mauceri and Meda, and we prove
\begin{theorem}\label{teobound-RieszX1}
$\nabla \Ls^{-1/2}$ is bounded from $X^1(\gamma)$ to $L^1(\gamma)$.
\end{theorem}
In the second part of the paper we provide a new proof of the weak type $(1,1)$ of $\nabla \Ls^{-1/2}$ (Theorem~\ref{11OU}) shorter and simpler than those appearing up to now in the literature. This is obtained by suitably combining some ideas of Pérez and Soria~\cite{PerSor} with some techniques introduced by Garc\'ia-Cuerva, Mauceri, Sj\"ogren and Torrea~\cite{GMST} and the same authors and Meda~\cite{GMMST}. Except for an elementary result~\cite[Lemma 4.4]{GMMST} and the theory developed in~\cite{GMST} for ``local'' Calder\'on-Zygmund operators, which can be considered an adaptation of the classical Calder\'on-Zygmund theory to the Gaussian setting in a certain neighbourhood of the diagonal of $\R^n\times \R^n$, our proof is self-contained. 

\smallskip

In the remaining of this section, we fix the notation and introduce the Riesz transform and some spectral multipliers of $\Ls$ which will be of use. The definition of $X^1(\gamma)$ and the boundedness of $\nabla \Ls^{-1/2}$ from $X^1(\gamma)$ to $L^1(\gamma)$ is the object of Section~\ref{Sec:X1:OU}, while the new proof of the weak type $(1,1)$ of $\nabla \Ls^{-1/2}$ occupies Section~\ref{Sec:11OU}. Further details are given at the beginning of these two sections.

\subsection{Integral kernels}\label{intKernels}
The $L^2(\gamma)$-spectrum of the Ornstein-Uhlenbeck operator $\Ls$ is the set of nonnegative integers $\{0,1,\dots\}$, and its eigenfunctions are (tensor product of) Hermite polynomials. Its spectral resolution $(\mathcal{P}_k)$, $k=0,1,\dots$ is the family of orthogonal projectors of $L^2(\gamma)$ onto the subspaces generated by the Hermite polynomials. It is also well known that $\Ls$ is the infinitesimal generator of the Mehler semigroup $e^{-t\Ls}$, whose kernel $M_t$ with respect to the Lebesgue measure\footnote{Given a bounded operator $T$ on $L^2(\gamma)$, we say that a distribution $K_T$ on  $\R^n \times \R^n$ is its Schwartz kernel \emph{with respect to the Lebesgue measure} if
\[
Tf(x)= \int_{\R^n} K_T(x,y) f(y)\, \dd y
\]
for a.e.\ $x\in \R^n$.} is
\[
M_t(x,y)= \frac{1}{\pi^{n/2}(1-e^{-2t})^{n/2}}\exp\left( -\frac{|e^{-t}x -y|^2 }{1-e^{-2t}}\right).
\]
We refer the reader e.g.\ to~\cite{Sjogren} for further details.

For every $z\in \C$, with a slight abuse of notation, we define
\begin{equation}\label{Lz}
\Ls^z = \sum_{k=1}^\infty k^z \mathcal{P}_k, \qquad \Dom(\Ls^z)= \bigg\{f\in L^2(\gamma)\colon \sum_{k=1}^\infty k^{2\Re z} \|\mathcal{P}_k\|_2^2<\infty\bigg\}.
\end{equation}
If $\Re z <0$, $\Ls^z$ is bounded on $L^2(\gamma)$ and $\Dom(\Ls^z)=L^2(\gamma)$. If $\Re z \geq 0$, observe that $C_c^\infty \subset \Dom(\Ls^z)$ by the decomposition $\Ls^z = \Ls^{z-N} \Ls^N$ where $N=[\Re z] +1$. \par

Let $\Pi_0$ be the orthogonal projection
\[
\Pi_0\colon L^2(\gamma) \to \ker(\Ls)^\perp=\left\{f\in L^2(\gamma)\colon \int f\, \dd \gamma=0 \right\}.
\]
In terms of the spectral resolution $(\Pc_k)$ of $\Ls$, $\Pi_0 = I-\Pc_0$, since 
\[ \Pc_0\colon L^2(\gamma) \to \ker (\Ls)= \C, \quad \Pc_0 f = \int f\, \dd \gamma.\]
Observe moreover that $\Ran(\Ls)$ is closed, since $\Ls$ is closed and has spectral gap. Thus $\ker(\Ls)^\perp =\Ran(\Ls)$. We shall denote the space $\Pi_0 L^2(\gamma)$ also by $L^2_0(\gamma)$. Observe that 
\begin{equation*}\label{eqdagger}
\Ls \Ls^{-1} f = \Pi_0 f \quad \forall \, f\in L^2(\gamma), \qquad \Ls^{-1} \Ls f = \Pi_0 f \quad \forall\, f\in \Dom(\Ls),
\end{equation*}
and in particular
\begin{equation}\label{eqdagger}
\Ls \Ls^{-1} f = f \quad \forall \, f\in L^2_0(\gamma), \qquad \Ls^{-1} \Ls f = f \quad \forall\, f\in \Dom(\Ls)\cap L^2_0(\gamma).
\end{equation}
For every $b\in \R\setminus \N$, the kernel of the operator $\Ls^{b}$ with respect to the Lebesgue measure is
\begin{align*}
K_{\Ls^{b}}(x,y)&= \frac{1}{\Gamma(-b)}\int_0^\infty t^{-b-1}(M_t(x,y)-\pi^{-n/2}e^{-|y|^2})\, \dd t\nonumber
\\&= \frac{1}{\Gamma(-b)} \int_0^1 (-\log r)^{-b-1} (M_{(-\log r)}(x,y)-\pi^{-n/2}e^{-|y|^2} )\, \frac{\dd r}{r}
\end{align*}
where we used the change of variables $t=-\log r$. See e.g.~\cite{GMST, GMST2}. In particular, for every $j=1,\dots,n$, the kernel of the operator $\nabla \Ls^{1/2}$ is
\begin{align}\label{nucleoDL12OU}
K_{\nabla\Ls^{1/2}}(x,y)= -\pi^{\frac{n+1}{2}} e^{|x|^2-|y|^2} \int_0^1 \frac{(-\log r)^{-3/2}}{(1-r^2)^{(n+2)/2}} (rx-y) e^{-\frac{|x-ry|^2}{1-r^2}}\, \dd r,
\end{align}
while the kernel of the Riesz transform associated with $\Ls$, i.e.\ the operator $\nabla \Ls^{-1/2}$, is
\begin{align}
K_{\nabla \Ls^{-1/2}}(x,y)&=-\frac{2}{\pi^{(n+1)/2}} \int_0^\infty \frac{t^{-1/2} e^{-t}}{(1-e^{-2t})^{(n+2)/2}}(e^{-t}x -y) e^{-\frac{|e^{-t}x - y|^2}{1-e^{-2t}}}\, \dd t \nonumber \\&= -\frac{2}{\pi^{(n+1)/2}} e^{|x|^2-|y|^2} \int_0^1 \frac{(-\log r)^{-1/2}}{(1-r^2)^{(n+2)/2}} (rx-y) e^{-\frac{|x-ry|^2}{1-r^2}}\, \dd r 
\label{KernelRieszOU2}
\end{align}
again by the change of variables $t=-\log r$. Both the kernels in~\eqref{nucleoDL12OU} and~\eqref{KernelRieszOU2} are with respect to the Lebesgue measure.

\smallskip

All throughout the paper, we shall use the letters $c$ and $C$ to denote constants, not necessarily equal at different occurrences. For any quantity $A$ and $B$, we write $A\lesssim B$ by meaning that there exists a constant $c>0$ such that $A\leq c \,B$. If $A\lesssim B$ and $B\lesssim A$, we write $A\approx B$ .

\section{The Hardy space $X^1(\gamma)$}\label{Sec:X1:OU}
In a recent series of papers Mauceri, Meda and Vallarino~\cite{MMVHardy,MMVAtomic, MMV, MMV-Harmonic} developed a theory of Hardy-type spaces on certain noncompact manifolds of infinite volume, to obtain endpoint estimates for imaginary powers and Riesz transforms associated with the Laplace-Beltrami operator of the manifold. Though in a rather different context, we shall adapt their Hardy spaces to the Gaussian setting (thus of \emph{finite} volume) and to the Ornstein-Uhlenbeck operator.

The atoms we shall use are classical atoms supported in (dilations of) ``hyperbolic'' balls, which will be called \emph{admissible}. We inherit such atoms and terminology from~\cite{MM}. When talking about balls, we always mean \emph{Euclidean} balls. If $B$ is a ball, $c_B$ will stand for its center and $r_B$ for its radius. For every positive integer $k$ and ball $B$, we shall write $kB$ to denote the ball with same center $c_B$ and radius $k\, r_B$.
\begin{definition}\label{admissible-ballOU}
We call \emph{admissible ball} a ball $B$ of center $c_B$ and radius $r_B \leq \min (1,1/|c_B|)$. The family of all admissible balls will be denoted by $\mathcal{B}_1$.
\end{definition}

\begin{definition}\label{quasi-harmonic-lambda}
Let $\Omega$ be a bounded open set and $K$ be a compact set.
\begin{itemize}
\item We denote by $q^2(\Omega)$ the space of all functions $u\in L^2(\Omega)$ such that $\Ls u$ is constant on $\Omega$, and by $q^2(K)$ the space of functions on $K$ which are the restriction to $K$ of a function in $q^2(\Omega')$ for some bounded open $\Omega'\supset K$; 
\item we denote by $h^2(\Omega)$ the space of all functions $u\in L^2(\Omega)$ such that $\Ls u=0$ on $\Omega$, and by $h^2(K)$ the space of functions on $K$ which are the restriction to $K$ of a function in $h^2(\Omega')$ for some bounded open $\Omega'\supset K$.
\end{itemize}
The spaces $h^2(\Omega)^\perp$ and $q^2(\Omega)^\perp$ are the orthogonal complements of $h^2(\Omega)$ and $q^{2}(\Omega)$ in $L^2(\Omega,\gamma)$, respectively. The spaces $h^2(K)^\perp$ and $q^2(K)^\perp$ are the orthogonal complements in $L^2(K,\gamma)$.
\end{definition}

We now introduce the atomic Gaussian Hardy space $X^1(\gamma)$. The reader should compare our definitions with those of~\cite{MMVAtomic}.

\begin{definition}\label{atom-X1}
An \emph{$X^1$-atom} is a function $a\in L^2(\gamma )$, supported in a ball $B\in \mathcal{B}_1$, such that
\begin{itemize}
\item[\rmi] $\|a\|_{L^2(\gamma )} \leq \gamma (B)^{-1/2}$,
\item[\rmii] $a\in q^2(\bar{B})^\perp$.
\end{itemize}
\end{definition}
\begin{definition}\label{def:X1mu}
The Hardy space $X^{1}(\gamma )$ is the space
\begin{equation*}
X^{1}(\gamma )\coloneqq\big\{ f \in L^1(\gamma ) \colon \mbox{$f= \sum_j \mu_j a_j$}, \mbox{ $a_j$ $X^1$-atom }, (\mu_j) \in \ell^1\big\}
\end{equation*}
endowed with the norm
\[
\|f\|_{X^1(\gamma)}\coloneqq \inf\; \{\| (\mu_j)\|_{\ell^1}\colon\mbox{$f= \sum_j \mu_j a_j$}, \mbox{ $a_j$ $X^1$-atom} \}.
\]
\end{definition}
If $B\in \mathcal{B}_1$, the functions in $q^2(\bar{B})$ will be referred to as (Gaussian) \emph{quasi-harmonic functions} on $B$.

Observe that the space $X^1(\gamma)$ is strictly contained in the Hardy space $H^1(\gamma)$ introduced by Mauceri and Meda~\cite{MM}. Indeed, the atoms defining $H^1(\gamma)$ are supported on admissible balls and satisfy property \rmi~of Definition~\ref{atom-X1}, but have only \emph{zero integral}, a much weaker condition than \rmii~of the same definition. In this sense, the space $X^1(\gamma)$ may be inserted in the framework of the theory developed by Mauceri and Meda~\cite{MM} for the Gauss measure or more generally by Carbonaro and the same authors~\cite{CMMfinita} in the setting of metric measure spaces. However, it is worth mentioning that our understanding of the space $X^1(\gamma)$ is still far from being complete and it will be the object of further investigations.

Our proof of Theorem~\ref{teobound-RieszX1} follows the same order of ideas of~\cite[Theorem 5.3]{MMV}.

\subsection{Support preservation on atoms} A key point of the proof of~\cite[Theorem 5.3]{MMV} is that the inverse of the Laplace-Beltrami operator preserves the support of atoms. In the following proposition, we prove that $\Ls^{-1}$ (suitably defined, recall~\eqref{Lz}) shares the same behaviour on $X^1$-atoms. Its proof will occupy the remainder of this subsection.
\begin{proposition}\label{support:presOU}
For every $X^1$-atom $a$ supported in an admissible ball $B$, $\supp \Ls^{{-1}} a \subseteq \bar{B}$ and
\[\|\Ls^{-1} a\|_{L^2(\gamma )}\leq r_B^2\, \gamma (B)^{-1/2}. \]
\end{proposition}
For every ball $B$, in the same spirit of~\cite{MMV}, we introduce two operators $\Ls _B$ and $\Ls_{B, \Dir}$, defined as the restriction of $\Ls$ (in the distributional sense) to 
\[\Dom(\Ls _B)\coloneqq \{ f\in \Dom(\Ls) \colon \, \supp f\subseteq \bar{B}\},\]
\[\Dom(\Ls _{B,\Dir})\coloneqq \{ f\in W^{1,2}_0(B,\gamma )\colon \, \Ls f \in L^2(B,\gamma )\}\]
respectively. Here $W^{1,2}_0(B,\gamma )$ denotes the closure of $C_c^\infty(B)$ with respect to the graph norm of the gradient $\nabla$ on $L^2(B,\gamma)$. We shall also use the space $W^{2,2}_0(B,\gamma )$ which is the closure of $C_c^\infty(B)$ with respect to the graph norm of $\Ls$. Notice that, since $\gamma $ and $\gamma^{-1}$ are bounded on any compact set, $W^{2,2}_0(B,\gamma )= W^{2,2}_0(B)$ as vector spaces, with equivalent norms.

It is well known (cf.~\cite[Theorem 10.13]{Grigor'yan}) that $\Ls_{B,\Dir}$ has purely discrete spectrum. We denote by $\lambda_\Dir^{\gamma}(B)$ its first eigenvalue.

We begin by proving the following proposition. Its proof is essentially the same as~\cite[Proposition 3.5]{MMV-Harmonic}, but avoids to use the existence of \emph{global} quasi-harmonic functions. We include all the details for the ease of the reader. 
\begin{lemma}\label{prop4OU}
Let $B$ be a ball. Then
\begin{itemize}
\item[\emph{(1)}] both $h^2(B)$ and $q^2(B)$ are closed in $L^2(B)$; 
\item[\emph{(2)}] $\Ls$ is a Banach space isomorphism between $\Dom(\Ls_B)$ and $h^2(\bar B)^\perp$;
\item[\emph{(3)}] $h^2(\bar B)^\perp = h^2(B)^\perp$;
\item[\emph{(4)}] $q^2(\bar{B})^\perp= q^2(B)^\perp$.
\end{itemize}
\end{lemma}

\begin{proof}
In the whole proof, $B$ will be a fixed ball. 

(1) Let $\psi_B \in C_c^\infty\cap L^2_0(\gamma)$ be such that $\psi_B|_B \equiv 1_B$. Then
\[q^2(B) = h^2(B) \oplus \C (\Ls^{-1} \psi_B)|_B\]
and thus it is enough to prove that $h^2(B)$ is closed, since it is a subspace of $q^2(B)$ of codimension one. Now let $(v_k)$ be a sequence in $h^2(B)$ converging to $v$ in $L^2(B)$. Then $\Ls v_k$ converges to $\Ls v$ in the sense of distributions in $B$, thus $\Ls v =0$ in $B$ and hence $v\in h^2(B)$.

(2) We first show that $\Ls (\Dom(\Ls_B)) \subseteq h^2(\bar{B})^\perp$. Let then $f\in \Dom(\Ls_B)$, and $v\in h^2(\bar{B})$. Let $\tilde v$ be a smooth function with compact support which is harmonic in an open neighbourhood of $\bar{B}$ and satisfies $\tilde v|_{\bar{B}} =v$. Then
\[\int_{\bar{B}} v\Ls f\, \dd \gamma = \int_{\R^n} \tilde v \Ls f\, \dd \gamma = \int_{\R^n} \Ls \tilde v f\, \dd \gamma =0\]
since $\supp f\subseteq \bar{B}$ and $\Ls \tilde v$ is zero in a neighbourhood of $\bar{B}$.

$\Ls$ is injective on $\Dom(\Ls_B)$, since if $f\in \Dom(\Ls_B)$ and $\Ls f=0$, then $f$ is constant and has compact support; thus, $f=0$.

We now prove that $\Ls$ maps $\Dom(\Ls_B)$ onto $h^2(\bar{B})^\perp$. In order to do this, let $v\in h^2(\bar{B})^\perp$ and let $\tilde v$ be the extension of $v$ to a null function outside $\bar{B}$. Let $f = \Ls^{-1} \tilde v$.

By definition, $f\in \Dom(\Ls)$. Since the constant function $1$ is in $h^2(\bar{B})$, moreover, $\int \tilde v\, \dd \gamma =\int v \, \dd \gamma =0$; thus $\tilde v \in L^2_0(\gamma)$ and by~\eqref{eqdagger}, $\Ls f = \tilde v$. 

Let now $\phi \in C_c^\infty(\bar{B}^c)\cap L^2_0(\gamma)$. Then, by~\eqref{eqdagger}, $\Ls \Ls^{-1} \phi = \phi$ which is identically zero on a neighbourhood of $\bar{B}$. Therefore, $\Ls^{-1}\phi \in h^2(\bar{B})$ and hence
\begin{align*}
(\phi,f)_{L^2(\gamma )}&=(\Ls \Ls^{-1} \phi,f)_{L^2(\gamma )}\\&=(\Ls ^{-1} \phi,\Ls f)_{L^2(\gamma )} =\int_B (\Ls^{-1} \phi)v\dd \gamma=0.
\end{align*}
This implies that there exists a constant $c$ such that $f=c$ on $\bar{B}^c$. Thus, $g\coloneqq f-c$ is such that $g\in \Dom(\Ls)$, $\supp g \subseteq \bar{B}$ and $\Ls g= \Ls f= \tilde v$.

We have proved that $\Ls$ is a bijection between $\Dom(\Ls_B)$ and $h^2(\bar B)^\perp$. Since $\Ls $ is continuous from $\Dom(\Ls_B)$ to $h^2(\bar B)^\perp$, its inverse is also continuous by the closed graph theorem, and $\Ls$ is then a Banach space isomorphism.

(3) We use a simple adaptation of~\cite[Theorem 3.4, $(ii)\Rightarrow (i)$]{MMV-Harmonic}. The obvious inclusion $h^2(\bar B) \subseteq h^2(B)$ leads to 
\[h^2(B)^\perp \subseteq h^2(\bar B)^\perp.\]
As for the converse inclusion, observe first that, if $W^{2,2}(\R^n)_{\bar{B}}$ denotes the functions in $W^{2,2}(\R^n)$ with support in $\bar{B}$, then $W^{2,2}(\R^n)_{\bar{B}} = W^{2,2}_0(B)$ by~\cite[Chapter 5.5, Theorem 2]{Evans} (or~\cite[Theorem 5.29]{AdamsFournier}). Moreover, since $B$ is bounded and has finite measure, $\Dom(\Ls_B)= \Dom(\Delta_B)$. Therefore
\begin{equation}\label{eqdomainsOU}
\Dom(\Ls_B)= \Dom(\Delta_B)= W^{2,2}(\R^n)_{\bar B}= W^{2,2}_0(B)= W^{2,2}_0(B, \gamma ).
\end{equation}
Let now $v\in h^2(\bar B)^\perp$ and $\tilde v$ be the extension of $v$ which vanishes on $\bar B^c$. By (2) there exists $f\in \Dom(\Ls_B)$ such that $\Ls f = \tilde v$, and by~\eqref{eqdomainsOU} there exists a sequence $(\phi_k) \subset C_c^\infty(B)$ converging to $f$ in the graph norm of $\Ls $. Thus, if $g\in h^2(B)$ and $\tilde g$ is its trivial extension to $\R^n$,
\begin{align*}
\int_{B}vg\, \dd\gamma& =(\tilde v,\tilde g)= (\Ls f, \tilde g)= \lim_k(\Ls \phi_k, \tilde g)= \lim_k(\phi_k, \Ls \tilde g)
\end{align*}
which vanishes because $\Ls \tilde g=0$ on $B$. Thus, $v\in h^2(B)^\perp$.

(4) We prove that, for every ball $B$, $\overline{q^2(\bar{B})}=q^2(B)$. The inclusion $\subseteq $ follows easily, since the obvious inclusion $q^2(\bar{B})\subseteq q^2(B)$ leads to
\[\overline{q^2(\bar{B})}\subseteq \overline{q^2(B)}=q^2(B),\]
the last equality being true by (1).

To prove the converse inclusion $\supseteq$, let $v\in q^2(B)$ so that $\Ls v=c$ on $B$ for some constant $c$. Let $g\in L^2_0(\gamma)$ be such that $g=c$ on a neighbourhood of $\bar B$. Let $q= \Ls^{-1} g$, so that $\Ls q=g$ by~\eqref{eqdagger}, $q\in q^2(\bar B)$ and $v-q \in h^2(B)$. By (3) and (1)
\begin{equation}\label{hehchiuso}
\overline{h^2(\bar{B})}= \overline{h^2(B)}=h^2(B),
\end{equation}
and thus there exists a sequence $(h_k) \subseteq h^2(\bar{B})$ such that $h_k \to v-q$ in $L^2(B)$, and then $h_k +q \to v$ in $L^2(B)$. Therefore $v\in \overline{q^2(\bar{B})}$.
\end{proof}

\begin{lemma}\label{dir-rangoOU}
Let $B$ be a ball. Then
\begin{itemize}
\item[\emph{(1)}] $\Ls _B \subset \Ls _{B,\Dir}$;
\item[\emph{(2)}] $\Ran(\Ls _B)= h^2(B)^\perp $.
\end{itemize}
\end{lemma}
\begin{proof}
We adopt the same strategy of~\cite[Proposition 3.1 (i)]{MMV}.

(1) Let $f\in \Dom(\Ls_B)$. Then
\[\|\nabla f\|_{L^2(\gamma )} \leq \|f\|_{L^2(\gamma )} \|\Ls f\|_{L^2(\gamma )} <\infty\]
and since $\supp f\subseteq \bar{B}$, $f\in W^{1,2}(B)$. Since $f=0$ on the complement of $\bar{B}$ and the boundary of $B$ is smooth, the trace of $f$ on the boundary of $B$ is zero. Thus $f\in W^{1,2}_0(B)$ by a classical result (see e.g.~\cite[Chapter 5.5, Theorem 2]{Evans}). Thus $\Dom(\Ls_B) \subset \Dom(\Ls_{B,\Dir})$.

(2) First of all, $\Ran(\Ls _B)$ is closed in $L^2(B)$, since it is closed in $L^2(\gamma )$, because $\Ls$ has spectral gap and is closed. Thus, to prove the inclusion $\supseteq$ of (2) it suffices to show that $\Ran(\Ls _B)^\perp \subseteq h^2(B)$. 

Let $g\in \Ran(\Ls _B)^\perp$. Then
\[0= \int_B (\Ls \psi )g \, \dd\gamma = \langle \gamma \psi, \Ls g\rangle \qquad \forall \psi \in C_c^\infty(B) \]
in the sense of distributions on $B$. Hence $\Ls g= 0$ on $B$, namely $g\in h^2(B)$.

We finally prove the inclusion $\subseteq$. Since $h^2(B)=\overline{h^2(\bar{B})}$ by Lemma~\ref{prop4OU}, (3), it is enough to prove that $\Ran(\Ls _B)$ is orthogonal to $h^2(\bar{B})$. Let then $f\in \Dom(\Ls _B)$, $g\in h^2(\bar{B})$ and let $\tilde g$ be any extension of $g$ to all $\R^n$, such that $\tilde g\in \Dom(\Ls)$. Thus, since $\supp (\Ls f)\subseteq \bar{B}$
\[(\Ls _B f, g)_{L^2(B,\gamma )}= (\Ls f, \tilde g)_{L^2(\gamma )} = (f, \Ls \tilde g)_{L^2(\gamma )}=0\]
because $\supp f \subseteq \bar{B}$ and $\Ls \tilde g$ vanishes on a neighbourhood of $\bar{B}$.
\end{proof}

\begin{proof}[Proof of Proposition~\ref{support:presOU}]
Let $a$ be an $X^1$-atom. By Lemmata~\ref{prop4OU}, (4) and~\ref{dir-rangoOU}, (2) we get
\[a\in q^2(\bar{B})^\perp = q^2(B)^\perp \subset h^2(B)^\perp = \Ran(\Ls _B).\]
Therefore, there exists $f\in \Dom(\Ls _B)$ such that $\Ls_{B,\Dir} f=\Ls _B f = a$, the first equality by Lemma~\ref{dir-rangoOU}, (1). Thus $\supp \,(\Ls_{B,\Dir}^{-1} a)= \supp f \subseteq \bar{B}$. Moreover, $\Ls ^{-1} a = \Ls_{B,\Dir}^{-1}a$. Thus
\begin{equation}\label{LB-1aOU}
\|\Ls^{-1} a\|_2 = \|\Ls^{-1}_{B,\Dir}a\|_2 \leq \frac{1}{\lambda_{\Dir}^\gamma(B)} \|a\|_2\leq \frac{\gamma (B)^{-1/2}}{\lambda_{\Dir}^\gamma(B)}. 
\end{equation}
It then remains to estimate $\lambda_\Dir^\gamma(B)$. Recall that on $\R^n$ we have the usual Faber-Krahn inequality for the Laplacian 
\begin{equation}\label{FKclassicOU}
\lambda_1(B) \geq C |B|^{-2/n}
\end{equation}
where $\lambda_1(B)$ is the first eigenvalue of the Dirichlet Laplacian (see e.g.~\cite[(14.5)]{Grigor'yan}). Then, by the minmax principle~\cite[Theorem 10.18]{Grigor'yan} and the equivalence of the Lebesgue measure and $\gamma$ on $B$
\[\lambda_\Dir^{\gamma}(B) = \inf_{\phi \in C_c^\infty(B) \setminus \{0\}} \frac{\int_B |\nabla \phi|^2(x) \gamma (x)\, \dd x}{\int_B |\phi|^2(x) \gamma (x)\, \dd x} \geq c \inf_{\phi \in C_c^\infty(B)\setminus \{0\}}\frac{\int_B |\nabla \phi|^2(x)\, \dd x}{\int_B |\phi|^2(x)\, \dd x}=c \lambda_1(B) \]
for some $c>0$, independent of $B\in \mathcal{B}_1$. Then, by~\eqref{FKclassicOU}
\begin{equation*}
\lambda_\Dir^{\gamma}(B) \geq c \lambda_1(B) \geq C |B|^{-2/n}\geq c r_B^{-2}
\end{equation*}
since $|B|\approx r_B^n$. This together with~\eqref{LB-1aOU} completes the proof.
\end{proof}

\subsection{Proof of Theorem~\ref{teobound-RieszX1}}\label{Sec:RieszX1}
\begin{lemma}\label{Lemma4BOU}
For every ball $B\in \mathcal{B}_1$ and every $f\in L^1(\gamma)$ with $\supp f\subseteq \bar{B}$,
\[\|\nabla \Ls^{1/2} f\|_{L^1((4B)^c, \gamma)}\lesssim r_B^{-2}\|f\|_{L^1(B,\gamma)}.\]
\end{lemma}
\begin{proof}
By~\eqref{nucleoDL12OU}
\begin{align*}
\|\nabla \Ls^{1/2} f\|_{L^1((4B)^c, \gamma)} &\lesssim \int_{(4B)^c}\int_B \int_0^1 \frac{|r x-y|e^{-\frac{|x-ry|^2}{1-r^2}}}{(1-r^2)^{\frac{n+2}{2}}(-\log r)^{3/2}} \, \dd r |f(y)| \, \dd \gamma (y) \, \dd x\\& = \int_B I(y)|f(y)|\, \dd \gamma(y)\,
\end{align*}
where for $y\in B$
\begin{equation}\label{integrale4B}
I(y)= \int_0^1 \frac{1}{(1-r^2)^{n/2 +1}(-\log r)^{3/2}} \int_{(4B)^c} |rx-y|e^{-\frac{|x-ry|^2}{1-r^2}}\, \dd x \, \dd r.
\end{equation}
The proof will then be complete if we can show $I(y)\lesssim r_B^{-2}$ for every $y\in B$. We split $I(y)$ into $I_1(y) + I_2(y)$ according to the splitting $(0,1)= (0,1/2] \cup (1/2,1)$. Thus
\begin{align*}
I_1(y) \lesssim \int_0^{1/2} \frac{1}{(-\log r)^{3/2}}\int_{(4B)^c} |rx-y|e^{-|x-ry|^2}\, \dd x \, \dd r.
\end{align*}
We make the change of variables $x-ry=v$ in the inner integral and then extend the integration domain to $\R^n$. This yields
\begin{align*}
I_1(y) \lesssim \int_0^{1/2} \frac{1}{(-\log r)^{3/2}}\int_{\R^n} |rv+(r^2-1)y|e^{-|v|^2}\, \dd v \, \dd r.
\end{align*}
Now observe that, since $|y|\leq |c_B|+r_B \leq 2/r_B$ by the admissibility condition of the ball $B$,
\[|rv+(r^2-1)y| \leq r|v|+|y|\leq r|v|+\frac{2}{r_B}\lesssim\frac{|v|+1}{r_B}\]
since $r_B\leq 1$, and hence
\[
I_1(y) \leq \frac{C}{r_B} \int_0^{1/2} \frac{1}{(-\log r)^{3/2}}\int_{\R^n} (|v|+1)e^{-|v|^2}\, \dd v \, \dd r \leq \frac{C}{r_B}.
\]
Therefore, \emph{a fortiori}, $I_1(y)\lesssim r_B^{-2}$. Before looking at $I_2(y)$, we observe that for every $r\in (1/2,1)$
\begin{align*}
|rx-y|\leq |x-ry| + (1-r^2)|y|,
\end{align*}
since $rx-y= r(x-ry) -(1-r^2)y$. Hence
\begin{align*}
I_2(y) &\lesssim \int_{1/2}^1 \frac{1}{(1-r^2)^{n/2 +2 }}\int_{(4B)^c} \bigg[\frac{|x-ry|}{\sqrt{1-r^2}} + \sqrt{1-r^2}|y|\bigg]e^{-\frac{|x-ry|^2}{1-r^2}}\, \dd x\, \dd r.
\end{align*}
By using the inequalities $se^{-s^2}\lesssim e^{-s^2/2}$ for $s>0$ and $e^{-s^2}\leq e^{-s^2/2}$, we get
\begin{align*}
I_2(y) & \lesssim \int_{1/2}^1 \frac{1 + \sqrt{1-r^2}|y|}{(1-r^2)^{2 }}(1-r^2)^{-n/2}\int_{(4B)^c}e^{-\frac{|x-ry|^2}{2(1-r^2)}}\, \dd x\, \dd r.
\end{align*}
We now separate the cases when $r_{B,y}\geq 1$ and $r_{B,y}<1$, where (see \cite[Lemma 7.1]{MM} for the notation)
\[
r_{B,y}= r_{B}/(2|y|).
\]
If $r_{B,y}\geq 1$, by \cite[Lemma 7.1, (i) and (iii)]{MM}
\begin{align*}
I_2(y) \leq \int_{1/2}^1 \frac{1 + \sqrt{1-r^2}|y|}{(1-r^2)^{2 }} \varphi \bigg(\frac{r_B}{\sqrt{1-r^2}}\bigg)\, \dd r
\end{align*}
which yields, after the change of variables $r_B / \sqrt{1-r^2}=s$,
\[I_2(y) \lesssim \frac{1}{r_B^2}\int_0^\infty (s+r_B |y|)\varphi(s)\, \dd s \lesssim \frac{1}{r_B^2}\int_0^\infty (s+1)\varphi(s)\, \dd s = \frac{C}{r_B^2}\]
since $r_B|y|\leq C$. 

If $r_{B,y}<1$, we split $(1/2,1)=(1/2, 1-r_{B,y}] \cup (1-r_{B,y},1)$ and $I_2(y) = I_2^1(y)+I_2^2(y)$ accordingly. By \cite[Lemma 7.1, (ii)]{MM}, $I_2^2(y)$ can be treated exactly as we did in the case $r_{B,y}\geq 1$, so we concentrate on $I_2^1(y)$ only. By the change of variable $x-ry=v$ in the inner integral, we get
\begin{align*}
I_2^1(y)& \lesssim \int_{1/2}^{1-r_{B,y}} \frac{1+\sqrt{1-r^2}|y|}{(1-r^2)^2} \int_{\R^n} e^{-|v|^2}\, \dd v\, \dd r
\\&\lesssim \int_{1/2}^{1-r_{B,y}} \frac{1+\sqrt{1-r}|y|}{(1-r)^2}\, \dd r \lesssim \frac{1}{r_{B,y}} + \frac{|y|}{\sqrt{r_{B,y}}}\lesssim \frac{1}{r_B^2}
\end{align*}
since $|y|\lesssim 1/r_B$ and by the definition of $r_{B,y}$.
\end{proof}
\begin{proof}[Proof of Theorem~\ref{teobound-RieszX1}]
We follow the same line as~\cite[Theorem 5.3]{MMV} to prove that
\[
\sup\; \{ \|\nabla \Ls^{-1/2} a\|_1 \colon \, \mbox{$a$ is an $X^1$-atom}\}<\infty.
\]
Since $\nabla \Ls^{-1/2}$ is of weak type $(1,1)$, this implies the boundedness $X^1(\gamma)\to L^1(\gamma)$ by a classical argument~\cite[p.\ 95]{Grafakos}.

Let $a$ be an $X^1$-atom supported in an admissible ball $B$. Since
\[\|\nabla \Ls^{-1/2} a\|_{L^1(\gamma)}= \|\nabla \Ls^{-1/2} a\|_{L^1(4B, \gamma)} + \|\nabla \Ls^{-1/2} a\|_{L^1((4B)^c,\gamma)}\]
it is enough to estimate the two summands separately. First, by Cauchy-Schwarz
\[\|\nabla \Ls^{-1/2} a\|_{L^1(4B,\gamma)} \leq \gamma(4B)^{1/2}\|\nabla \Ls^{-1/2} a\|_{L^2(4B,\gamma)}\lesssim \|a\|_2\,  \gamma(4B)^{1/2} \leq C\]
where we used the boundedness of $\nabla \Ls^{-1/2}$ on $L^2(\gamma)$, the size property of $a$ and the local doubling property of $\gamma $. As for the second summand, we write 
\[\nabla \Ls^{-1/2} a= \nabla \Ls^{1/2}\Ls^{-1} a \]
by the spectral theorem. By Proposition~\ref{support:presOU}, $\supp \Ls^{-1} a \subseteq \bar{B}$. Therefore, by Lemma~\ref{Lemma4BOU}, Cauchy-Schwarz inequality and Proposition~\ref{support:presOU} respectively
\[\|\nabla \Ls^{-1/2} a\|_{L^1((4B)^c,\gamma)} \lesssim r_B^{-2}\|\Ls^{-1} a\|_{L^1(B,\gamma)} \lesssim r_B^{-2}\gamma(B)^{1/2}\|\Ls^{-1} a\|_{L^2(\gamma)} \leq C.\]
The proof is complete.
\end{proof}

\section{Weak type $(1,1)$}\label{Sec:11OU}
Since $\gamma$ is locally, but not globally doubling, it is a standard procedure to split $\R^n\times \R^n$ as the union of a neighbourhood of the diagonal and of its complement, and to split accordingly the kernels of the operators. Thus, for $\delta>0$ we define
\begin{equation}\label{LandGOU}
N_\delta \coloneqq \left\{(x,y)\in \R^n\times \R^n \colon |x-y|\leq \frac{\delta}{1+|x|+|y|}\right\}, \qquad G\coloneqq N_1^c.
\end{equation}
We shall call both $N_1$ and $N_2$ the \emph{local regions} and $G$ the \emph{global region}, in analogy with~\cite{GMST}. We shall also fix once and for all a smooth function $\chi$ such that 
\[\chi_{N_1}\leq \chi \leq \chi_{N_2},\qquad |\nabla_x \chi(x,y)|+|\nabla_y \chi(x,y)|\leq \frac{C}{|x-y|}\quad \mbox{for } x\neq y,\] and define
\[
K_{\nabla \Ls^{-1/2},\loc}\coloneqq \chi K_{\nabla \Ls^{-1/2}}, \qquad K_{\nabla \Ls^{-1/2},\glob}\coloneqq K_{\nabla \Ls^{-1/2}}-K_{\nabla \Ls^{-1/2},\loc}.
\]
We shall denote the operators with kernel $K_{\nabla \Ls^{-1/2},\loc}$ and $K_{\nabla \Ls^{-1/2},\glob}$ by $\nabla \Ls^{-1/2}_\loc$ and $\nabla \Ls^{-1/2}_\glob$ respectively. Of course
\begin{equation}\label{loca+globOU}
\nabla \Ls^{-1/2}= \nabla \Ls^{-1/2}_\loc + \nabla \Ls^{-1/2}_\glob.
\end{equation}
Therefore, in order to prove the weak type $(1,1)$ of $\nabla \Ls^{-1/2}$ it will be enough to prove the weak type $(1,1)$ of both $\nabla \Ls^{-1/2}_\loc$ and $\nabla \Ls^{-1/2}_\glob$. The proof for $\nabla \Ls^{-1/2}_\loc$ (Proposition~\ref{proplocalrieszOU}) is rather standard, since by a general result (see~\cite[Theorem 2.7]{GMST}) this can be reduced to proving weak type $(1,1)$ boundedness of some classical Calder\'on-Zygmund operators. The key proof is then that concerning $\nabla \Ls^{-1/2}_\glob$.

To do this, we prove that $K_{\nabla \Ls^{-1/2},\glob}$ is controlled by a kernel $\overline M$ which arises naturally from the global part of the Mehler maximal operator. This idea is not completely new, as it comes from the paper~\cite{PerSor} of Pérez and Soria. Their proof is based on the following facts: (1) providing a kernel $\overline{K}$ \emph{equivalent} to the Mehler maximal kernel in the global region~\cite[Proposition 2.1]{MenPerSor} (2) proving that $\overline{K}$ is the kernel of an operator of weak type $(1,1)$~\cite[Theorem 2.3]{MenPerSor}, and (3) proving that the kernel of the Riesz transform is controlled by $\overline{K}$ in the global region~\cite[Proposition 2.2]{PerSor}. Though we follow the same order of ideas, the kernel $\overline{M}$ that we obtain (Proposition~\ref{maxglobalOU}) controls only from above the Mehler maximal kernel, except in a certain region (see Remark~\ref{ossbetaminore1OU}) where they are equivalent. This greatly simplifies the proofs, for the weak type $(1,1)$ of the operator associated with $\overline{M}$ can be easily deduced (Lemma~\ref{weakMbarOU}) by a kernel obtained by García-Cuerva, Mauceri, Meda, Sjögren and Torrea~\cite{GMMST}. Finally, we prove that $\overline{M}$ controls also the kernel of the Riesz transform in the global region (Proposition~\ref{rieszglobalOU}). Our proofs use a useful rescaling of the Mehler kernel introduced by García-Cuerva, Mauceri, Sjögren and Torrea in~\cite{GMST}. 

\smallskip

We begin by fixing the notation and obtaining some elementary results that will be used later on. Then, in Subsection~\ref{SecMehOU} we shall show that the kernel $\overline M$ arises naturally from the study of the Mehler maximal operator in the global region, and prove the weak type $(1,1)$ of its associated operator. Finally, in Subsection~\ref{SecRiesz} we shall prove Theorem~\ref{11OU} by proving the weak type $(1,1)$ of both $\nabla \Ls^{-1/2}_\loc$ and $\nabla \Ls^{-1/2}_\glob$. 

\smallskip

For $x,y\in \R^n$ set
\[
\alpha\coloneqq |x-y||x+y|,\qquad \beta\coloneqq \frac{|x-y|}{|x+y|}, \qquad \eta(x,y)\coloneqq e^{-\frac{|x|^2}{2}+\frac{|y|^2}{2}- \frac{|x-y||x+y|}{2}}.
\]
We also set $\theta=\theta(x,y)$ to be the angle between $x$ and $y$, and $\theta'$ the angle between $y-x$ and $y+x$. Observe that $\beta<1$ if and only if $(x,y)>0$. The results contained in the following lemma will be used all throughout the remainder of paper. Though their proofs are elementary, we provide all the details.
\begin{lemma}\label{lemmaprelOU}
Let $(x,y)\in \R^n$. Then
\begin{itemize}
\item[\emph{(1)}] if $(x,y)\in G$ and $\beta<1$, then $\alpha\geq 1/4$.
\item[\emph{(2)}] if $(x,y)\in G$, then $|x-y|\geq \frac{1}{2}(1+|x|)^{-1}$.
\item[\emph{(3)}] $|x \pm y|\geq |x|\sin\theta$. In particular, $\alpha\geq |x|^2\sin^2\theta$.
\item[\emph{(4)}] ${-|x|^2} +{|y|^2} -{|x-y||x+y|}\leq 0$.
\item[\emph{(5)}] $-\frac{|x|^2}{2} + \frac{|y|^2}{2} - \frac{|x+y||x-y|}{2}=\frac{-2|x|^2|y|^2 \sin^2\theta}{|x-y||x+y|(1+\cos \theta')}.$
\end{itemize}
\end{lemma}
\begin{proof}
To prove (1), first assume $|x|+|y|\leq 1$. Then
\[
|x-y||x+y| \geq |x-y|^2 \geq \frac{1}{(1+|x|+|y|)^2}\geq \frac{1}{4}.
\]
Since $\beta<1$, $|x+y|\geq |x|$ and $|x+y|\geq |y|$. Observe moreover that the function $t\mapsto t/(1+t)$ is increasing. Thus, if $|x|+|y|>1$,
\[
|x+y||x-y|\geq \frac{|x+y|}{1+|x|+|y|}\geq \frac{1}{2}\frac{|x|+|y|}{1+|x|+|y|}\geq \frac{1}{4}.
\]
The proof of (2) is shown in~\cite[pg.\ 225]{GMMST}. The point (3) holds since $|x|\sin \theta$ is the length of the projection of $x\pm y$ on the hyperplane orthogonal to $y$. To be more explicit,
\[
|x\pm y|^2 -|x|^2\sin^2\theta = |x|^2\cos^2\theta +|y|^2 \pm 2|x||y|\cos\theta = (|x|\cos\theta\pm|y|)^2\geq 0.
\]
As for (4), just observe that
\[
-{|x|^2} +{|y|^2} -{|x-y||x+y|} = (y+x,y-x) -|x-y||x+y|\leq0
\]
by H\"older's inequality. Equivalently, one can see (4) as a consequence of (5) which is just a computation.
\end{proof}
\subsection{The Mehler Maximal Operator}\label{SecMehOU}
It is well known that the Mehler maximal operator $\mathcal{M}^*$, namely the operator with kernel
\[
M^*(x,y) \coloneqq \sup_t M_t(x,y)
\]
with respect to the Lebesgue measure, is of weak type $(1,1)$. See, for example,~\cite{MenPerSor} and~\cite{GMMST}. Here, we provide a different proof of the weak type $(1,1)$ of its global part (Proposition~\ref{maxglobalOU} below), from which the following kernel arises naturally.
\begin{definition}\label{Msegnato}
Define
\begin{equation*}
\overline{M}(x,y) \coloneqq e^{|x|^2-|y|^2} \left( \frac{|x+y|}{|x-y|}\right)^{n/2} e^{-\frac{|x|^2}{2}+\frac{|y|^2}{2}- \frac{|x-y||x+y|}{2}} \Psi(x,y)\chi_{G}(x,y),
\end{equation*}
where
\[
\Psi(x,y)= \max\left(1,\frac{1}{\alpha^{n/2}}\right).
\]
\end{definition}
Though the following result plays no role in the proof of Theorem~\ref{11OU}, we provide its proof for it highlights the origin of the kernel $\overline{M}$.
\begin{proposition}\label{maxglobalOU}
For every $(x,y)\in G$, $M^*(x,y) \lesssim \overline M(x,y)$.
\end{proposition}

\begin{proof}
First of all, we perform the change of variable
\begin{equation}\label{tausOU}
\tau(s)\coloneqq \log \frac{1+s}{1-s}
\end{equation}
introduced in~\cite{GMMST}. Then
\[M^*(x,y)= \sup_{0<s<1} M_{\tau(s)}(x,y).\]
An easy computation shows that
\begin{align*}
M_{\tau(s)}(x,y) &= \frac{(1+s)^n}{(4s)^{n/2}} e^{-\frac{|y|^2}{2} + \frac{|x|^2}{2} -\frac{1}{4}(s|x+y|^2+\frac{1}{s}|x-y|^2)} \\&= e^{|x|^2-|y|^2} e^{-\frac{|x|^2}{2}+\frac{|y|^2}{2}}\frac{(1+s)^n}{(4s)^{n/2}} e^{-\frac{1}{4}(s|x+y|^2+\frac{1}{s}|x-y|^2)},
\end{align*}
so that
\[
M^*(x,y) = e^{|x|^2-|y|^2} \eta(x,y) \sup_{0<s<1}\frac{(1+s)^n}{(4s)^{n/2}} e^{-\frac{1}{4}(s|x+y|^2+\frac{1}{s}|x-y|^2 -2|x-y||x+y|)}.
\] 
We now make the substitution $s/\beta=\sigma$ in the supremum, and get
\begin{equation*}
M^*(x,y)=2^{-n}\, e^{|x|^2-|y|^2} \left( \frac{|x+y|}{|x-y|}\right)^{n/2} \eta(x,y)\sup_{0<\sigma< 1/\beta}\frac{\left(1+\sigma \beta\right)^n}{\sigma^{n/2}} e^{-\frac{1}{4}\alpha\varphi(\sigma)},
\end{equation*}
where
\[\varphi(\sigma)\coloneqq \sigma + \frac{1}{\sigma}-2 = \frac{(\sigma-1)^2}{\sigma}.\]
It remains then to estimate the supremum. The first observation is that the contribution of the term $(1+\sigma\beta)^n$ can be neglected, since $1\leq (1+\sigma \beta)^n \leq 2^n$ for $\sigma\in (0,1/\beta)$. Thus
\[
\sup_{0<\sigma< 1/\beta}\frac{\left(1+\sigma \beta\right)^n}{\sigma^{n/2}} e^{-\frac{1}{4}\alpha\varphi(\sigma)} \approx \sup_{0<\sigma< 1/\beta}\frac{1}{\sigma^{n/2}} e^{-\frac{1}{4}\alpha\varphi(\sigma)}.
\]
If $\beta< 1$, we have $\alpha\geq 1/4$ by Lemma~\ref{lemmaprelOU} (1). Thus 
\[
\sup_{0<\sigma< 1/\beta}\frac{1}{\sigma^{n/2}} e^{-\frac{1}{4}\alpha\varphi(\sigma)}\leq \sup_{0<\sigma<\infty}\frac{1}{\sigma^{n/2}} e^{-\frac{1}{16}\varphi(\sigma)}\leq C.
\]
Let now $\beta\geq1$, and observe that the function 
\[\sigma\mapsto \frac{1}{\sigma^{n/2}}e^{-\frac{1}{4}\alpha \varphi(\sigma)} \]
is increasing in the interval $(0,\sigma_0)$ and decreasing in $(\sigma_0,\infty)$, where
\[
\sigma_0 = \frac{\sqrt{4n^2+\alpha^2}-2n}{\alpha}.
\]
Therefore
\[
\sup_{0<\sigma< 1/\beta}\frac{1}{\sigma^{n/2}} e^{-\frac{1}{4}\alpha\varphi(\sigma)}\leq \frac{1}{\sigma_0^{n/2}} e^{-\frac{1}{4}\alpha \varphi(\sigma_0)}\lesssim \max\left(1,\frac{1}{\alpha^{n/2}}\right).
\]
In other words, we have proved that (see also Remark~\ref{ossbetaminore1OU} below)
\begin{equation}\label{supPsiOU}
\sup_{0<\sigma< 1/\beta}\frac{1}{\sigma^{n/2}} e^{-\frac{1}{4}\alpha\varphi(\sigma)}\lesssim \Psi(x,y)
\end{equation}
and this completes the proof.
\end{proof}
\begin{remark}\label{ossbetaminore1OU}
If $(x,y)\in G$ and $\beta <1$, then $\alpha\geq 1/4$ by Lemma~\ref{lemmaprelOU}, (1). Thus, $1\leq \Psi(x,y) \leq 2^n$ for every $(x,y)\in G$. In this case then $\overline M$ controls from above and below the kernel $M^*$, with absolute constants. This was first shown in~\cite[Proposition 2.1]{MenPerSor}. 
\end{remark}

As stated above, we now prove the weak type $(1,1)$ of the operator whose kernel is $\overline{M}$. By Proposition~\ref{maxglobalOU}, this implies the weak type $(1,1)$ of the global part of the Mehler maximal operator $\mathcal{M}^*_\glob$, which is the operator with kernel $M^*(1-\chi)$ with respect to the Lebesgue measure. 
\begin{lemma}\label{weakMbarOU}
The operator with kernel $\overline M(x,y)$ with respect to the Lebesgue measure is of weak type $(1,1)$. In particular, $\mathcal{M}^*_\glob$ is of weak type $(1,1)$.
\end{lemma}

\begin{proof}
We only prove that, for $(x,y)\in G$,
\[
\overline M(x,y)\lesssim e^{|x|^2 - |y|^2} (1+|x|)^n \wedge (|x|\sin \theta)^{-n}
\]
or, equivalently, that
\[
\left( \frac{|x+y|}{|x-y|}\right)^{n/2} \eta(x,y)\Psi(x,y)\lesssim (1+|x|)^n \wedge (|x|\sin \theta)^{-n}.
\]
The conclusion will then follow by~\cite[Lemma 4.4]{GMMST}.

We first consider the inequality involving $(1+|x|)^n$. We consider the cases $\Psi(x,y)=1$ and $\Psi(x,y)= 1/\alpha^{n/2}$ separately.

\textbf{1.} If $\Psi(x,y)=1$, then by Lemma~\ref{lemmaprelOU}, (4) it is enough to prove that
\begin{equation}\label{eq1+xOU}
\left( \frac{|x+y|}{|x-y|}\right)^{n/2} \lesssim (1+|x|)^n.
\end{equation}
If $|y|\leq 2|x|$ then by Lemma~\ref{lemmaprelOU}, (2) we get
\[
\frac{|x+y|}{|x-y|}\leq \frac{|x|+|y|}{|x-y|}\lesssim |x|(1+|x|)\leq (1+|x|)^2.
\]
If instead $|y|> 2|x|$, we have
\[|x-y|\geq |y|-|x| \geq |y|/2, \qquad |x-y|\geq |y|-|x|\geq |x|\]
so that
\[
\frac{|x+y|}{|x-y|}\leq \frac{|x|}{|x-y|} + \frac{|y|}{|x-y|}\leq C
\]
and hence \emph{a fortiori}~\eqref{eq1+xOU} holds.

\textbf{2.} If $\Psi(x,y)=1/{\alpha^{n/2}}$
\[
\left( \frac{|x+y|}{|x-y|}\right)^{n/2} \eta(x,y) \Psi(x,y) = \frac{\eta(x,y) }{|x-y|^n} \leq \frac{1}{|x-y|^n}\lesssim(1+|x|)^n,
\]
again by Lemma~\ref{lemmaprelOU}, (4) and (2).

We then concentrate on the inequality involving $(|x|\sin\theta)^{-n}$. We again consider the cases $\Psi(x,y)=1$ and $\Psi(x,y)=1/{\alpha^{n/2}}$ separately.

\textbf{1'.} Let $\Psi(x,y)=1$, and observe that the function $0\leq u\mapsto u^{n/2}e^{-u}$ is bounded. Thus, by Lemma~\ref{lemmaprelOU} (5)
\begin{multline*}
\left(\frac{|x+y|}{|x-y|}\right)^{n/2} \eta(x,y)
= \left(\frac{|x+y|}{|x-y|}\right)^{n/2} e^{\frac{-2|x|^2|y|^2 \sin^2\theta}{|x-y||x+y|(1+\cos \theta')}}
\\ \leq \left(\frac{|x+y|^2(1+\cos \theta')}{2|x|^2|y|^2\sin^2\theta}\right)^{n/2}\\ = C (|x|\sin\theta)^{-n} \left(\frac{|x+y|^2(1+\cos\theta')}{|y|^2}\right)^{n/2}.
\end{multline*}
Therefore, it remains only to prove that
\[
\frac{|x+y|^2(1+\cos\theta')}{|y|^2}\leq C.
\]
If $|x|\leq 2|y|$ this is straightforward. Otherwise, note that
\begin{align*}
\frac{|x+y|^2(1+\cos\theta')}{|y|^2} = g_\theta(|x|^2/|y|^2),
\end{align*}
where
\[g_\theta(t)= (1+t +2\sqrt{t}\cos\theta) \left(1+\frac{1-t}{\sqrt{(1+t)^2-4t\cos^2\theta}}\right).\]
Finally, observe that the functions $g_\theta$ are bounded on $(4,\infty)$ uniformly in $\theta$.

\textbf{2'.} If $\Psi=1/{\alpha^{n/2}}$, observe that
\[
\left( \frac{|x+y|}{|x-y|}\right)^{n/2} \eta(x,y)\Psi(x,y) \leq \frac{1}{|x-y|^n} \lesssim \frac{1}{(|x|\sin\theta)^n}
\]
by Lemma~\ref{lemmaprelOU}, (4) and (3). This completes the proof.
\end{proof}

\subsection{Proof of Theorem~\ref{11OU}}\label{SecRiesz}
As already said, we treat separately the local and the global part of $\nabla \Ls^{-1/2}$. By means of~\eqref{loca+globOU}, Theorem~\ref{11OU} will be a consequence of Propositions~\ref{proplocalrieszOU} and~\ref{rieszglobalOU} below.

\smallskip

In order to treat the local part $\nabla \Ls^{-1/2}_\loc$, we shall need the following lemma.
\begin{lemma}\label{lemmalocalOU}
Let $\mu,\nu\geq 0$ be such that $\mu >\nu +1$. Then, for every $(x,y)\in N_2$, $x\neq y$
\[
R_{\mu,\nu}(x,y)\coloneqq \int_0^1 \frac{|rx-y|^\nu }{(1-r^2)^{\frac{n+\mu}{2}}}e^{-\frac{|rx-y|^2}{1-r^2}}\, \dd r \leq \frac{C}{|x-y|^{n+\mu-\nu-2}}.
\]
\end{lemma}

\begin{proof}
Assume $(x,y)\in N_2$ and $x\neq y$. Observe that
\begin{align*}
R_{\mu,\nu}(x,y)\lesssim \int_0^1 \frac{1}{(1-r^2)^{\frac{n+\mu-\nu}{2}}} e^{-\frac{1}{2}\frac{|rx-y|^2}{1-r^2}}\, \dd r
\end{align*}
since the function $s \mapsto s^\nu e^{-s^2/2}$ is bounded for every $s\geq 0$ and $\nu\geq 0$. Now observe that
\begin{align*}
|rx-y|^2 \geq |x-y|^2 -2(1-r)|x||x-y|\geq |x-y|^2 -4(1-r)
\end{align*}
where the last inequality holds since for all $(x,y)\in N_2$
\[
|x||x-y|\leq \frac{2|x|}{1+|x|+|y|}\leq 2.
\]
Thus
\begin{align*}\label{lemmalocal1}
R_{\mu,\nu}(x,y) &\lesssim \int_0^1 \frac{1}{(1-r^2)^{\frac{n+\mu-\nu}{2}}} e^{-\frac{1}{2}\frac{|x-y|^2}{1-r^2}}\, \dd r \lesssim \int_0^1 \frac{1}{(1-r)^{\frac{n+\mu-\nu}{2}}} e^{-c\frac{|x-y|^2}{1-r}}\, \dd r
\end{align*}
and by performing the change of variable $|x-y|^2/(1-r)=t$ we get
\[
R_{\mu,\nu}(x,y) \leq \frac{C}{|x-y|^{n+\mu-\nu-2}} \int_{0}^\infty t^{(n+\mu-\nu -4)/2}e^{-ct}\, \dd t\leq \frac{C}{|x-y|^{n+\mu-\nu-2}}
\]
where the last inequality holds since by assumption
\[
(n+\mu+\nu-4)/2 >(n-3)/2\geq -1.\qedhere
\]
\end{proof}
\begin{proposition}\label{proplocalrieszOU}
For every $j=1,\dots,n$, $\nabla \Ls^{-1/2}_\loc$ is of weak type $(1,1)$ .
\end{proposition}
\begin{proof}
Let $(x,y)\in N_2$, $x\neq y$. Observe that by~\eqref{KernelRieszOU2} and Lemma~\ref{lemmalocalOU},
\[\abs{K_{\nabla \Ls^{-1/2},\loc}(x,y)}\lesssim R_{3,1}(x,y)\chi(x,y)\lesssim|x-y|^{-n}\]
and
\begin{multline*}
\abs{\nabla_{x} K_{\nabla \Ls^{-1/2},\loc}(x,y)} + \abs{\nabla_{y} K_{\nabla \Ls^{-1/2},\loc}(x,y)}\\ \lesssim( R_{3,0}(x,y) + R_{5,2}(x,y) + R_{3,1}(x,y)|x-y|^{-1})\chi(x,y)\lesssim|x-y|^{-(n+1)}
\end{multline*}
for every $j=1,\dots,n$. Therefore, the conclusion follows by~\cite[Theorem 2.7]{GMST}.
\end{proof}

\begin{proposition}\label{rieszglobalOU}
For every $j=1,\dots,n$ and $(x,y)\in G$
\begin{equation}\label{rieszglobalOUeq}
|K_{\nabla \Ls^{-1/2}}(x,y)|\lesssim\overline{M}(x,y).
\end{equation}
In particular, $\nabla \Ls^{-1/2}_\glob$ is of weak type $(1,1)$ for every $j=1,\dots,n$.
\end{proposition}

\begin{proof}
Let $(x,y)\in G$, and observe first that for every $j=1,\dots,n$ 
\[
|K_{\nabla \Ls^{-1/2}}(x,y)|\lesssim \int_0^\infty \frac{e^{-t}}{(1-e^{-2t})^{(n+2)/2}}\frac{|e^{-t}x -y|}{\sqrt{1-e^{-2t}}} e^{-\frac{|e^{-t}x - y|^2}{1-e^{-2t}}}\, \dd t\eqqcolon R(x,y)
\]
since $t\geq (1-e^{-2t})/2$ for every $t\geq 0$. With the change of variables $t=\tau(s)$ (recall~\eqref{tausOU}) in the integral defining the kernel $R$, 
\begin{align*}
|K_{\nabla \Ls^{-1/2}}(x,y)|& \lesssim \int_0^1 \frac{1}{s^{(n+3)/2}} \abs{(1-s)x-(1+s)y} e^{-\frac{\abs{(1-s)x-(1+s)y}^2}{4s}}\, \dd s\\& = e^{|x|^2-|y|^2} \eta(x,y) \int_0^1 \frac{1}{s^{(n+3)/2}}\abs{(1-s)x-(1+s)y} e^{-\frac{1}{4}\alpha \varphi(s/\beta)}\, \dd s,
\end{align*}
where we used that $1+s\geq 1$ for every $s\in (0,1)$. Now make the change of variables $s/\beta =\sigma$ in the integral, which gives
\begin{multline*}
\int_0^1 \frac{1}{s^{(n+3)/2}}\abs{(1-s)x -(1+s)y} e^{-\frac{1}{4}\alpha \varphi(s/\beta)}\, \dd s \\ = \frac{1}{\beta^{n/2}} \int_0^{1/\beta} \frac{1}{\sigma^{(n+3)/2}}\frac{\abs{(1-\sigma\beta)x - (1+\sigma\beta)y}}{\sqrt{\beta}}e^{-\frac{1}{4}\alpha \varphi(\sigma)}\, \dd \sigma.
\end{multline*}
Observe moreover that
\begin{align*}
\frac{|(1-\sigma \beta)x -(1+\sigma\beta)y|}{\sqrt{\beta}} = \frac{|(x-y) -\sigma\beta(x+y)|}{\sqrt{\beta}} \leq \frac{|x-y|+\sigma|x-y|}{\sqrt{\beta}} = (1+\sigma)\sqrt{\alpha}.
\end{align*}
Therefore, we proved that for every $(x,y)\in G$
\begin{equation*}
|K_{\nabla \Ls^{-1/2}}(x,y)|\lesssim e^{|x|^2-|y|^2}\left(\frac{|x+y|}{|x-y|}\right)^{n/2} \eta(x,y) \sqrt{\alpha}\int_0^{1/\beta} \frac{(1+\sigma)}{\sigma^{(n+3)/2}}e^{-\frac{1}{4}\alpha \varphi(\sigma)}\, \dd \sigma.
\end{equation*}
It remains to prove that, if $(x,y)\in G$,
\begin{equation}\label{ineqPhi}
\sqrt{\alpha}\int_0^{1/\beta} \frac{(1+\sigma)}{\sigma^{(n+3)/2}}e^{-\frac{1}{4}\alpha \varphi(\sigma)}\, \dd \sigma \lesssim \Psi(x,y).
\end{equation}
Observe first that
\begin{multline*}
\sqrt{\alpha}\int_0^{1/\beta} \frac{1+\sigma}{\sigma^{(n+3)/2}}e^{-\frac{1}{4}\alpha \varphi(\sigma)}\, \dd \sigma\\ \leq \sqrt{\alpha}\sup_{0<\sigma<1/\beta} \left( \frac{1}{\sigma^{n/2}}e^{-\frac{1}{4}\alpha \varphi(\sigma)}\right)^{1-\frac{1}{n}} \int_0^{1/\beta}\frac{ (1+\sigma)}{\sigma^2} e^{-\frac{1}{4n}\alpha\varphi(\sigma)}\, \dd \sigma
\\ \lesssim \Psi(x,y)^{1-\frac{1}{n}} \sqrt{\alpha}\int_0^{1/\beta} \frac{(1+\sigma)}{\sigma^2} e^{-\frac{1}{4n}\alpha\varphi(\sigma)}\, \dd \sigma.
\end{multline*}
The last inequality holds by~\eqref{supPsiOU}. We now split the integral as
\[
\int_0^{1/\beta} \frac{1+\sigma}{\sigma^2} e^{-\frac{1}{4n} \alpha\varphi(\sigma)}\, \dd \sigma =\int_0^{\min(1,1/\beta)}\dots \, \dd \sigma + \int_{\min(1,1/\beta)}^{1/\beta}\dots\, \dd \sigma,
\] 
where we mean that the second integral is identically zero if $\beta\geq 1$. Since $\varphi$ is invertible in $(0,1)$ and $(1,\infty)$, it is invertible in both the integrals above, so that by the change of variables $\alpha \varphi(\sigma)=t$ we get
\begin{equation}\label{1nPsi1}
\sqrt{\alpha}\int_0^{\min(1,1/\beta)} \frac{1+\sigma}{\sigma^{2}} e^{-\frac{1}{4n} \alpha\varphi(\sigma)}\, \dd \sigma \leq \frac{1}{\sqrt{\alpha}} \int_0^{\infty}\frac{1}{1-\sigma_-(t)} e^{-\frac{t}{4n}}\, \dd t
\end{equation}
while
\begin{equation}\label{1nPsi2}
\sqrt{\alpha}\int_{\min(1,1/\beta)}^{1/\beta} \frac{1+\sigma}{\sigma^2} e^{-\frac{1}{4n} \alpha\varphi(\sigma)}\, \dd \sigma \leq \frac{C}{\sqrt{\alpha}} \int_0^{\infty}\frac{1}{\sigma_+(t)-1} e^{-\frac{t}{4n}}\, \dd t,
\end{equation}
where
\[\sigma_-(t)=1 -\frac{\sqrt{t^2 +4\alpha t}-t}{2\alpha},\qquad \sigma_+(t) =1 + \frac{\sqrt{t^2 +4\alpha t}+t}{2\alpha}.\]
It is not hard to see that
\begin{align*}
1-\sigma_-(t) = \frac{\sqrt{t^2 +4\alpha t} -t}{2\alpha}\geq C\min \left(1,\frac{\sqrt{t}}{\sqrt{\alpha}}\right)=\frac{C}{\sqrt{\alpha}}\min\left(\sqrt{\alpha},\sqrt{t}\right)
\end{align*}
by the inequality $\sqrt{1+z} -1\geq C \min (z,\sqrt{z})$. In other words,
\[
\frac{1}{1-\sigma_-(t)} \lesssim\sqrt{\alpha} \max\left(\frac{1}{\sqrt{\alpha}}, \frac{1}{\sqrt{t}}\right).
\]
Moreover
\[
\sigma_+(t)-1 = \frac{t+\sqrt{t^2+4\alpha t}}{2\alpha}\geq2 \frac{\sqrt{t}}{\sqrt{\alpha}}.
\]
Therefore, from~\eqref{1nPsi1}
\begin{align*}
\sqrt{\alpha}\int_0^{\min(1,1/\beta)} \frac{1+\sigma}{\sigma^{2}} e^{-\frac{1}{4n} \alpha\varphi(\sigma)}\, \dd \sigma &\lesssim \frac{1}{\sqrt{\alpha}} \int_0^{\infty}\frac{1}{1-\sigma_-(t)} e^{-\frac{t}{4n}}\, \dd t \\&\lesssim \int_0^\infty \max\left(\frac{1}{\sqrt{\alpha}}, \frac{1}{\sqrt{t}}\right) e^{-\frac{t}{4n}}\, \dd t\lesssim \max\left(\frac{1}{\sqrt{\alpha}}, 1\right),
\end{align*}
and from~\eqref{1nPsi2}
\[\sqrt{\alpha}\int_{\min(1,1/\beta)}^{1/\beta} \frac{1+\sigma}{\sigma^2} e^{-\frac{1}{4n} \alpha\varphi(\sigma)}\, \dd \sigma \lesssim \frac{1}{\sqrt{\alpha}} \int_0^{\infty}\frac{1}{\sigma_+(t)-1} e^{-\frac{t}{4n}}\, \dd t\leq C.\]
The proof of~\eqref{rieszglobalOUeq} is now complete. The weak type $(1,1)$ of $\nabla \Ls^{-1/2}_\glob$ is then a consequence of the straightforward observation that
$|K_{\nabla \Ls^{-1/2},\glob}|\leq |K_{\nabla \Ls^{-1/2}}|\chi_G$.
\end{proof}

\subsection*{Acknowledgements} It is a great pleasure to thank Giancarlo Mauceri and Stefano Meda for several fruitful discussions and their constant help and support.

	\noindent {\small
	\Addresses
	}
\end{document}